\documentclass[a4paper,11pt]{article}

\usepackage{ucs}
\usepackage[utf8]{inputenc} 
\usepackage[T1]{fontenc}

\usepackage{amsfonts}
\usepackage{amsmath}
\usepackage{amsthm}
\usepackage{amssymb}
\usepackage{mathrsfs} 
\usepackage{ae, aecompl}
\usepackage{graphicx}
\usepackage[a4paper, portrait]{geometry}
\usepackage{sectsty}
\usepackage{lmodern}

\usepackage{enumerate}
\date{}
\input xy
\xyoption{all}

\newcommand*{\qp}{\ensuremath{\mathbb{Q}_p}}                     
\newcommand*{\zp}{\ensuremath{\mathbb{Z}_p}}

\DeclareMathOperator{\Hom}{Hom}
\DeclareMathOperator{\emorph}{End}

\DeclareMathOperator{\GL}{GL}

\begin{document}

\title{Endomorphism algebras of admissible $p$-adic representations of $p$-adic Lie groups}
\author{Gabriel Dospinescu\footnote{CMLS, \'Ecole Polytechnique, {\it Adresse e-mail :} {\ttfamily gabriel.dospinescu@math.polytechnique.fr}}, Benjamin Schraen\footnote{CNRS-UVSQ, {\it Adresse e-mail :} {\ttfamily benjamin.schraen@uvsq.fr}}}
\maketitle

\begin{abstract}
\footnotesize

 Building on recent work of Ardakov and Wadsley, we prove Schur's lemma for absolutely irreducible admissible $p$-adic Banach space
  (respectively locally analytic) representations of $p$-adic Lie groups.
We also prove finiteness results for the endomorphism algebra of an irreducible admissible
representation.
 \end{abstract}

\newtheoremstyle{theorem}{11pt}{11pt}{\slshape}{}{\bfseries}{.}{.5em}{}
\newtheoremstyle{note}{11pt}{11pt}{}{}{\bfseries}{.}{.5em}{}

\theoremstyle{plain}
  \newtheorem{theorem}[subsection]{Theorem}
  \newtheorem{proposition}[subsection]{Proposition}
  \newtheorem{lemma}[subsection]{Lemma}
  \newtheorem{corollary}[subsection]{Corollary}

\theoremstyle{definition}
  \newtheorem{definition}[subsection]{Definition}
  \newtheorem{notation}{Notation}

\theoremstyle{remark}
  \newtheorem{example}[subsection]{Example}
  \newtheorem{remark}[subsection]{Remark}
  \newtheorem{remarks}[subsection]{Remarks}
  \newtheorem{exercise}[subsection]{Exercise}
  \newtheorem{conjecture}[subsection]{Conjecture}

\section{Introduction}

  In a series of papers \cite{ST1}, \cite{ST2}, \cite{ST3}, \cite{ST4}, Schneider and Teitelbaum defined good abelian categories of representations of $\qp$-analytic groups on locally convex $p$-adic vector spaces, by imposing suitable "admissibility" conditions. These representations
 play a crucial role in the, still in its infancy, $p$-adic local Langlands correspondence: deep results of Berger, Breuil, Colmez, Emerton, Kisin and Pa\v{s}k\={u}nas (see \cite{BB}, \cite{C2}, \cite{C5}, \cite{KiAst}, \cite{Pa}, to cite only a few) show that the study of irreducible unitary admissible Banach space representations for $\rm{GL}_2(\qp)$ comes down to the study of
 irreducible continuous Galois representations $\rho:\rm{Gal}(\overline{\qp}/\qp)\to \rm{GL}_2(\overline{\qp})$.
 This can be seen as an evidence that these representations are more subtle than their classical analogues
 (admissible smooth representations or Harish-Chandra modules). For instance,
 a basic result such as Schur's lemma was unknown in the $p$-adic Banach or locally analytic setting, while
 the corresponding results for admissible smooth representations or Harish-Chandra modules are rather elementary.
 In this paper, we prove Schur's lemma in complete generality, for (not necessarily unitary) irreducible admissible
 Banach space and locally analytic representations of $\qp$-analytic groups. This
  was known only for unitary irreducible Banach space representations
   of $\rm{GL}_2(\qp)$, having a central character and, for technical reasons, one also needs to assume that $p>3$. This was proved by Pa\v{s}k\={u}nas in \cite{Pa} in a highly indirect way,
 by fully using the deformation-theoretic aspects of the $p$-adic local Langlands correspondence and massive computations of extensions between smooth mod $p$ representations.

   To be more precise, let $p$ be a prime number, let $L$ be a finite extension of
   $\qp$ and let $G$ be a $\qp$-analytic group of finite dimension
   (for example, the $\qp$-points of an algebraic group over $\qp$). A Banach space representation
   $\Pi$ of $G$ over $L$ is an $L$-Banach space on which $G$ acts by
   $L$-linear continuous automorphisms. If $H$ is an open compact subgroup of $G$, let $C(H,L)$ be the Banach space of $L$-valued continuous
   maps on $H$. We say that $\Pi$ is admissible
   if one can find
   a closed $H$-equivariant embedding of $\Pi$ into $C(H,L)^n$ for some $n\geq 1$ and some $H$.
   Schneider and Teitelbaum also defined the notion of locally $\qp$-analytic representations of
   $G$ and that of admissibility for these representations. For the rather technical definitions,
   see section \ref{prelim}. Our main result is then:

   \begin{theorem}\label{Schur}
    Let $\Pi$ be a topologically irreducible admissible Banach space (resp. locally $\qp$-analytic) representation
    of a $\qp$-analytic group $G$ over $L$. Then

    1) The algebra $\rm{End}_{L[G]}(\Pi)$ of continuous $G$-equivariant
    endomorphisms of $\Pi$ is a finite dimensional division algebra over $L$.

    2) $\Pi$ is absolutely irreducible (i.e. $\Pi\otimes_{L} L'$ is topologically irreducible
    for any finite extension $L'$ of $L$) if and only if $\rm{End}_{L[G]}(\Pi)=L$.

    3) There exists a finite Galois extension $L'/L$ and finitely many absolutely
    irreducible admissible $L'$-Banach space (resp. locally $\qp$-analytic) representations
    $\Pi_1,\Pi_2,...,\Pi_s$ such that $\Pi\otimes_L L'\simeq \Pi_1\oplus\Pi_2\oplus...\oplus \Pi_s$.
   \end{theorem}

   An obvious consequence of the theorem is that absolutely irreducible
   admissible Banach space (resp. locally $\qp$-analytic) representations of $G$ have a central character (resp. a central
   character and an infinitesimal character). Soon after the paper was written, Ardakov and Wadsley informed us that they also have a proof of this result,
   which, as far as we know, was unknown even for $\rm{GL}_2(\qp)$.
    The hardest step of the proof of theorem \ref{Schur} is the fact that any
      $f\in \rm{End}_{L[G]} (\Pi)$ is algebraic over $L$.
      This uses basic structural results of Schneider and Teitelbaum \cite{ST3} (building on the seminal
      work of Lazard \cite{Laz}) and most crucially a deep analogue of Quillen's lemma \cite{Qu} for almost commutative affinoid $L$-algebras, due to
      Ardakov and Wadsley \cite{AW}. The finiteness (as $L$-vector space) of
    $\rm{End}_{L[G]}(\Pi)$ is a rather tricky combination of Baire's lemma and of a
    delicate theorem of Kaplansky \cite{Ka} on PI-division algebras. The second and third
    part of the theorem are directly inspired from \cite{Pa}, where the case of Banach space
    representations of $\rm{GL}_2(\qp)$ is treated. The case of general Banach space representations
    follows precisely the same lines once the first part of the theorem is proved. The arguments
    for the case of locally analytic representations are however trickier.

     We end this introduction by stating a conjecture which lies probably very deep, but which is really one of the results that one
    would be happy to have in the theory of $p$-adic representations of $p$-adic Lie groups.

     \begin{conjecture}\label{conj}
     Let $\Pi$ be an absolutely irreducible admissible $L$-Banach space representation of $G$.
     Then the space $\Pi_{\rm an}$ of locally analytic vectors in $\Pi$ is a $G$-representation
     (topologically) of finite length, having only scalar continuous equivariant endomorphisms
     (hence having an infinitesimal character).
     \end{conjecture}

   The subtle point is that $\Pi_{\rm an}$ is not generally irreducible. This is already seen for
   parabolic inductions (\cite{ST5}). Much more delicate examples of this rather surprising situation come
   from the $p$-adic Langlands correspondence for $\rm{GL}_2(\qp)$ (see \cite{C1}, \cite{D3}, \cite{LXZ}).
   To our knowledge, there is not a single non abelian $p$-adic Lie group for which the conjecture is known.
  If $G=\rm{GL}_2(\qp)$ and if $\Pi$ has a stable lattice whose reduction has
  finite length (this is automatic if $p>3$ and if $\Pi$ is unitary, by a deep theorem of Pa\v{s}k\={u}nas \cite{Pa}),
  one can show the existence of an infinitesimal character for $\Pi_{\rm an}$
  by a very indirect Galois-theoretic argument (one needs to combine results of \cite{D1}, \cite{D2}, which build on \cite{C5} and \cite{Pa}). Also, Colmez proved that in this case all continuous equivariant endomorphisms are
  scalar (unpublished).

\subsection{Notations and conventions}

  We fix a finite extension $L$ of $\qp$ and we let $O_L$ be its ring of integers. The norm $| \cdot |$ on $L$ is normalized such that $|p|=p^{-1}$. For the notions of
  $\qp$-analytic group and uniform pro-$p$ group we refer the reader to \cite{DDMS}. For basic facts of
  nonarchimedean functional analysis, see \cite{Sch}.
If $\Pi$ is a continuous representation of
$G$ on some locally convex $L$-vector space, $\rm{End}_{L[G]}(\Pi)$ will denote the $L$-algebra of continuous
$L$-linear $G$-equivariant endomorphisms of $\Pi$.

\subsection{Acknowledgements}
 The first author would like to thank K. Ardakov for kindly and quickly
answering some questions concerning \cite{AW}. He would also like to thank
Ga\"{e}tan Chenevier and Pierre Colmez for useful discussions around the results
of this article.

  \section{Preliminaries} \label{prelim}

\subsection{Duality and admissible Banach representations}

  Let $G$ be a $\qp$-analytic group and let $H$ be a compact open subgroup of $G$.
A Banach space representation of $G$ over $L$ is an $L$-Banach space
  $\Pi$ endowed with a continuous action of $G$ (thus the map $G\times \Pi\to \Pi$ is continuous). Let
$O_L[[H]]$ be the the completed group algebra of $H$ over $O_L$, that is the inverse limit of
the group algebras $O_L[H/H_1]$, where $H_1$ runs over the ordered set of compact open normal subgroups
of $H$. By a classical result of Lazard (\cite{Laz} V.2.2.4), $O_L[[H]]$ is a (left and right) noetherian algebra.
There is a natural isomorphism of $L$-topological algebras
between the algebra of $L$-valued measures $D^c(H,L)$ on $H$ (i.e. the topological dual
of the space of continuous $L$-valued functions on $H$) and $L\otimes_{O_L} O_L[[H]]$.
 Following Schneider and Teitelbaum \cite{ST1}, an $L$-Banach space representation $\Pi$ of $G$ is called admissible
  if its topological dual $\Pi'$ is a finitely generated module over
  $L\otimes_{O_L} O_L[[H]]$. This does not depend on the choice of $H$ and is equivalent
  to the existence of a closed $H$-equivariant embedding $\Pi\to C(H,L)^{n}$ for some
  $n\geq 1$ and some open compact subgroup $H$ of $G$ (\cite{Em5}, theorem 5.1.15 i)). We will constantly use the following
  basic result of Schneider and Teitelbaum (\cite{ST1}, theorem 3.5).

\begin{theorem}\label{dual1}
 The duality functor $V\to V'$ induces an anti-equivalence of categories between the category of
admissible $L$-Banach space representations of $H$ and the category of finitely generated
$L\otimes_{O_L} O_L[[H]]$-modules.
\end{theorem}

 Theorem \ref{dual1} combined with the fact that $L\otimes_{O_L} O_L[[H]]$ is noetherian implies
 that the category of admissible Banach space representations of $G$ is abelian. In particular,
continuous equivariant maps between
admissible Banach representations of $G$ are strict, with closed image.
Hence, if $\Pi_1$ and $\Pi_2$ are (topologically) irreducible admissible Banach space representations of $G$, then any nonzero $f\in
\Hom_{L[G]}(\Pi_1,\Pi_2)$ is an isomorphism (the previous discussion shows that $f$ has closed image and the topological
irreducibility of $\Pi_1$ and $\Pi_2$ forces $f$ being injective with dense image. Thus $f$ is bijective and so an isomorphism,
by the open image and closed graph theorems).

\begin{remark}
 If $G$ is compact, then topological irreducibility for an admissible Banach space
representation $\Pi$ of $G$ corresponds to the algebraic simplicity of $\Pi'$ as $D^c(G,L)$-module (\cite{ST1}, corollary 3.6). On the other hand, if $G$ is no longer compact and $\Pi$ is an irreducible admissible Banach space
representation of $G$, then
$\Pi'$ has no reason to be a simple $D^c(H,L)$ module.

\end{remark}

\subsection{Coadmissible modules and locally analytic representations}\label{coadm}

 We recall in this section a few basic facts about distribution algebras of uniform pro-$p$ groups,
 following section $4$ of \cite{ST3}. We will freely use the notions of $p$-valuation and ordered
 basis, for which we refer the reader to loc.cit. Let $G$ be a $\qp$-analytic group and let $H$ be a uniform
  open pro-$p$ subgroup of $G$ (such a subgroup always exists by \cite{Laz}).
  By section 4.2 of \cite{DDMS}, one can find an ordered basis
 $h_1,h_2,...,h_d$ of $H$ and a $p$-valuation $\omega$ on $H$ such that
 $\omega(h_i)=1$ if $p>2$ (resp. $\omega(h_i)=2$ if $p=2$). We obtain a bijective
 global chart for the manifold $H$, sending $(x_1,x_2,...,x_d)\in \zp^d$ to
$h_1^{x_1}h_2^{x_2}...h_d^{x_d}$ (of course, this is not a group homomorphism). Let $b_i=h_i-1\in L[H]$ and, if
$\alpha\in\mathbb{N}^d$, let $b^{\alpha}=b_1^{\alpha_1}b_2^{\alpha_2}...
b_d^{\alpha_d}\in L[H]$. Let $\mathcal{D}(H,L)$ be the algebra of $L$-valued distributions
on $H$, as defined in section 2 of \cite{ST2}. The theory of Mahler expansions
shows that the elements of
$\mathcal{D}(H,L)$ have unique convergent expansions $\lambda=\sum_{\alpha\in\mathbb{N}^d} c_{\alpha}b^{\alpha}$,
where $c_{\alpha}\in L$ and $\lim_{|\alpha|\to\infty} |c_{\alpha}|\cdot r^{|\alpha|}=0$ for all $r\in (0,1)$
(here $|\alpha|=\alpha_1+...+\alpha_d$). If $r\in (p^{-1}, 1)\cap p^{\mathbb{Q}}$, let $\mathcal{D}_r(H,L)$ be the completion of $\mathcal{D}(H,L)$
with respect to the norm $||\lambda||_r=\sup_{\alpha} |c_{\alpha}|\cdot r^{|\alpha|}$ (thus elements of
$\mathcal{D}_r(H,L)$ are convergent series $\sum_{\alpha} c_{\alpha} b^{\alpha}$ such that
$|c_{\alpha}| r^{|\alpha|}\to 0$ as $|\alpha|\to\infty$). In the case $p=2$, we will prefer the equivalent family of norms defined by $||\lambda||_r=\sup_{\alpha} |c_{\alpha}|\cdot r^{2 |\alpha|}$ (see \cite[Prop. $2.1$]{Schmidt}; the reason for this choice
is that we will fully use results of \cite{ST3}, where one uses the norm
$||\lambda||_r=\sup_{\alpha} |c_{\alpha}| r^{\sum \alpha_i \omega(h_i)}$; with our convention
and our choice of $\omega$, the norms $||\cdot||_r$ in this paper agree with the norms
$||\cdot||_r$ in \cite{ST3}).

 One of the main results of \cite{ST3} (theorem 4.10; the hypothesis HYP used in loc.cit is precisely
 the hypothesis that $H$ is a uniform pro-$p$ group, see the remark before lemma 4.4)
is that the natural topological isomorphism of $L$-algebras $\mathcal{D}(H,L)\simeq \varprojlim_r
\mathcal{D}_r(H,L)$ exhibits $\mathcal{D}(H,L)$ as a nuclear Fr\'echet-Stein algebra. That is,
$\mathcal{D}_r(H,L)$ are (left and right) noetherian Banach algebras and the transition maps are (right) flat
and compact. In section 3 of \cite{ST3}, Schneider
and Teitelbaum define the notion of coadmissible module for any nuclear Fr\'echet-Stein algebra.
Let us recall that any such module $M$ is a projective limit of a family $(M_r)$ (with $r\in (p^{-1},1)\cap p^{\mathbb{Q}}$),
 each $M_r$ being
 a finitely generated $\mathcal{D}_r(H,L)$-module such that for $r < r'$, $M_r=\mathcal{D}_r(H,L)
\otimes_{\mathcal{D}_{r'}(H,L)} M_{r'}$. For any coadmissible module $M$ we have
$M_r=\mathcal{D}_r(H,L) \otimes_{\mathcal{D}(H,L)} M$ and
the maps $M\to M_r$ have dense images. Any coadmissible module is naturally a nuclear Fr\'echet space and all
linear maps between coadmissible modules are automatically continuous, with closed image. Finally, the category
of coadmissible modules if abelian.




A locally $\qp$-analytic representation of $G$ is a locally
 convex $L$-vector space of compact type $V$, endowed with a separately continuous action of $G$
such that orbits maps $g \mapsto g \cdot v$ are in $C^{\rm an}(G,V)$ (see \cite{ST2} for the definition of this
space of locally $\qp$-analytic functions on $G$ in this generality). Since the space is of compact type, the action of $G$ is
automatically continuous. If $V$ is such a representation and if $H$ is an open compact subgroup of
$G$, the action of $G$ extends to a structure of $\mathcal{D}(H,L)$-module on $V'$. We say
that $V$ is admissible if $V'$ is a coadmissible $\mathcal{D}(H,L)$-module for one (equivalently, any)
compact open subgroup $H$ of $G$.


  Finally, we need to recall a fundamental result of Schneider and Teitelbaum (\cite{ST3}, theorem 7.1) on the existence of locally
$\qp$-analytic vectors in admissible Banach space representations. This will play a crucial role in the next section.
Let $\Pi$ be an admissible $L$-Banach space representation of $G$. A vector $v\in \Pi$ is called locally $\qp$-analytic if
 the map $g\to g\cdot v$ is in $\mathcal{C}^{\rm {an}}(G,\Pi)$. The subspace $\Pi_{\rm an}$ of locally $\qp$-analytic
 vectors of $\Pi$ is naturally a $G$-representation on a space of compact type if we endow it with
  the topology induced by the embedding $\Pi_{\rm an}\to C^{\rm an}(G,\Pi)$,
  which sends a vector to its orbit map.

  \begin{theorem}\label{ST}
  Let $\Pi$ be an admissible $L$-Banach space representation of
$G$. Then $\Pi_{\rm an}$ is a dense subspace of $\Pi$. Moreover,
it is an admissible locally $\qp$-analytic representation of $G$ and for any open compact subgroup $H$ of $G$,
$\Pi_{\rm an}'$ is isomorphic, as $\mathcal{D}(H,L)$ coadmissible module to
$\mathcal{D}(H,L)\otimes_{L\otimes_{O_L} O_L[[H]]} \Pi'$.
  \end{theorem}

\section{Proof of the main theorem}

\subsection{Ardakov-Wadsley's version of Quillen's lemma}

  The purpose of this section is to prove the following technical result, which will play
  a crucial role in the proof of theorem \ref{Schur}. The main ingredients are a generalization of
  Quillen's lemma, due to Ardakov and Wadsley \cite{AW} and the computation of the graded ring of
  $\mathcal{D}_r(H,L)$, due to Schneider and Teitelbaum \cite{ST3}. If $L$ is a finite extension of
  $\qp$, we will denote by $O_L$ its ring of integers, by $k_L$ its residue field and by
  $\pi_L$ a uniformizer of $L$.

\begin{theorem}\label{corAW}
Let $H$ be a uniform pro-$p$ group and let $r \in (p^{-1},1) \cap p^{\mathbb{Q}}$. Let $f$ be an endomorphism
of a finitely generated $\mathcal{D}_r(H,L)$-module $M$ such that each nonzero $g\in L[f]$ is invertible on $M$.
Then $f$ is algebraic over $L$.
\end{theorem}

\begin{proof}

  \textbf{Step 1} Choose a finite Galois extension $L'/L$ with absolute ramification index $e'$ such that
  $r=p^{-\frac{a}{e'}}$ (resp. $r^2=p^{-\frac{a}{e'}}$ if $p=2$)
  for some integer $a$. Let $A=\mathcal{D}_r(H,L')$, which is a noetherian $L'$-Banach algebra
  for the norm $||-||_r$ (which takes values in $p^{\frac{1}{e'}\mathbb{Z}}$, by our choice of
  $L'$). Consider the filtration $F^{\cdot} A$ on $A$, induced by the norm (so
  $F^s A$ is the space of those $\lambda\in A$ such that $||\lambda||_r\leq p^{-s}$). As $||-||_r$ takes values in $p^{\frac{1}{e'}\mathbb{Z}}$, we have
  $\cup_{s>0} F^s A=\pi_{L'} F^0 A$. Thus, if $\rm{gr}^{\cdot}$ denotes the graded
  ring of $A$ with respect to the filtration $F^{\cdot}A$, then $\rm{gr}^0(A)$ is naturally isomorphic to
  $F^0 A/ \pi_{L'} F^0 A$. The proof of lemma 4.8 of \cite{ST3} (which heavily depends theorem 4.5 of loc.cit)
  shows that we have an isomorphism of rings $$\rm{gr}^0(A)\simeq k_{L'}[u_1,u_2,...,u_d],$$ where
  $d=\dim H$. Therefore $F^0 A/ \pi_{L'} F^0 A$ is naturally a polynomial algebra over
  $k_{L'}$.

  \textbf{Step 2} Using corollary 8.6 of \cite{AW}, we obtain the following result: if $N$ is a finitely generated $A$-module and if $f\in \rm{End}_A(N)$ is an endomorphism
  such that all nonzero $g\in L'[f]$ are invertible on $N$, then $f$ is algebraic over $L'$. Indeed, by step 1 the algebra
  $A$ satisfies the hypotheses of loc.cit.

  \textbf{Step 3} Finally, let $f$ be as in the statement of the theorem and suppose that $f$ is not algebraic over
  $L$. Then for each non zero polynomial $P \in L[X]$, the endomorphism $P(f)$ of $\emorph_{\mathcal{D}_r(H,L)}(M)$ is invertible. If $Q \in L'[X]$ is non zero, then $P=\prod_{\sigma \in \mathrm{Gal}(L' /L)} \sigma(Q)$ is a non zero polynomial in $L[X]$, hence $P(f)$ is invertible on $M$, and $P(f \otimes \mathrm{Id})$ is invertible on $M \otimes_L L'$. As $P(f \otimes \mathrm{Id})= \prod_{\sigma \in \mathrm{Gal}(L'/L)} \sigma(Q)(f \otimes \mathrm{Id})$ we conclude that $Q(f \otimes \mathrm{Id})$ is invertible on $M \otimes_L L'$. By step 2, $f \otimes \mathrm{Id}$ is algebraic over $L'$, contradicting the fact that
   for each non zero polynomial $Q \in L'[X]$, $Q(f \otimes \mathrm{Id})$ is invertible, in particular non zero. The result follows.
\end{proof}

\begin{remark}
1) The proof of corollary 8.6 of \cite{AW} is fairly difficult, but the result they prove is much stronger
and applies in much more general situations. However, it would be very nice to have a more elementary proof
of the corresponding statement in our situation, namely: let $A$ be a Banach $L$-algebra with unit ball
$F^0 A$. Suppose that $F^0 A/\pi_L F^0A$ is a polynomial algebra over $k_L$. If $M$ is a finitely generated
$A$-module and if $f\in \rm{End}_A(M)$ is such that any nonzero $g\in L[f]$ is invertible on $M$, then
$f$ is algebraic over $L$.

2) Completed group algebras do not satisfy the hypotheses of corollary 8.6 of \cite{AW}, so when proving
Schur's lemma for admissible Banach space representations we will have to use Schneider's and Teitelbaum's results
on locally analytic vectors.
\end{remark}

\subsection{ Algebraicity of $\rm{End}_{L[G]}(\Pi)$}

  The following proposition is one of the key points of the paper. It is essentially
a combination of all results recalled in the previous chapter.

\begin{proposition} \label{alg}
  Let $\Pi$ be an irreducible admissible $L$-Banach space (resp. locally $\qp$-analytic) representation
  of a $\qp$-analytic group $G$. Then any $f\in \rm{End}_{L[G]} (\Pi)$ is algebraic
  over $L$.
   \end{proposition}

 \begin{proof}
 Choose an open uniform pro-$p$-subgroup $H$ of $G$ and fix $r\in (p^{-1},1)\cap p^{\mathbb{Q}}$ such that
  $(\Pi_{\rm an}')_r \neq 0$ (resp. $\Pi'_r\ne 0$ in the locally $\qp$-analytic case). Such an $r$ exists by theorem \ref{ST} in the Banach case and for obvious reasons in the case of locally $\qp$-analytic representations.

   Suppose first that $\Pi$ is a Banach space representation.
  By theorem \ref{dual1}, any $f\in \rm{End}_{L[G]} (\Pi)$ induces an endomorphism $f'$ of the finitely generated $D^c(H,L)$-module $\Pi'$. The natural isomorphism
 $$ (\Pi_{\rm an}')_r\simeq \mathcal{D}_r(H,L)\otimes_{D(H,L)} \Pi_{\rm an}'\simeq\mathcal{D}_r(H,L)\otimes_{D^c(H,L)}
\Pi',$$ obtained using theorem \ref{ST}, allows us to extend $f'$ by $\mathcal{D}_r(H,L)$-linearity to an element $f'_r$ of  $\rm{End}_{ \mathcal{D}_r(H,L)}
 (\Pi_{\rm an}')_r$. This yields an anti-homomorphism
$\rm{End}_{L[G]}(\Pi)\to\rm{End}_{\mathcal{D}_r(H,L)}  (\Pi_{\rm an}')_r$.
For locally analytic representations, one still has such an anti-homomorphism,
essentially by definition of coadmissible modules.

Let $f \in \rm{End}_{L[G]}(\Pi) \backslash \{ 0 \}$ and $A=L[f]\subset \rm{End}_{L[G]}(\Pi)$. As $\Pi$ is irreducible, the discussion following theorem \ref{dual1} (resp.
analogous properties of the category of admissible locally $\qp$-analytic representations)
shows that any nonzero $g\in A$ is invertible on $\Pi$.
As $g\to g'_r$ is anti-multiplicative, it follows that $g'_r$ is an automorphism of the finitely generated
$\mathcal{D}(H,L)_r$-module $(\Pi_{\rm an}')_r$ (resp. $\Pi'_r$). Theorem \ref{corAW} implies that $f'_r$ is algebraic over
 $L$, so there exists a nonzero polynomial $P \in L[X]$ such that $P(f'_r)=0$. It is easy to see that this forces
 $P(f)=0$ (otherwise $P(f)$ would be invertible and so $P(f'_r)$ would also be invertible).
The result follows.

\end{proof}

  The following proposition is heavily inspired by lemmas 4.1 and 4.2 of \cite{Pa}.
It is however more technically involved in the locally analytic case. Its proof would be
much easier (and identical to that of loc.cit) if we knew that an admissible locally analytic
representation contains an irreducible closed subrepresentation. For admissible Banach space representations,
this is an immediate consequence of the fact that any decreasing sequence of closed subrepresentations
is eventually constant (which itself follows easily by duality and the fact that $\mathcal{D}^c(H,L)$ is noetherian),
but this probably does not happen in the locally analytic case.

\begin{proposition}
 Let $\Pi$ be an irreducible admissible Banach space (resp. locally analytic) $L$-representation of
 $G$. Then $\Pi$ is absolutely irreducible if and only if $\rm{End}_{L[G]}(\Pi)=L$.
\end{proposition}

\begin{proof}
 Suppose that $\Pi$ is absolutely irreducible and let $f\in \rm{End}_{L[G]}(\Pi)$.
 Let $P\in L[X]$ be a polynomial killing $f$ (it exists by the previous proposition)
 and let $L'$ be a splitting field of $P$. There exists $c\in L'$ such that
 $f\otimes 1-c$ is not invertible on $\Pi\otimes L'$. As $\Pi\otimes L'$ is irreducible,
 it follows that $f\otimes 1=c$. Using the action of $\rm{Gal}(L'/L)$, it is easy to see
 that $c\in L$ and so $f=c$ (see \cite{Pa}, lemma 4.1 for details).

  The hard point is the converse implication, which is an easy consequence of the following

 \begin{lemma}
  Let $L'/L$ be a finite Galois extension and let $\Pi$ be an irreducible admissible Banach space
  (resp. locally analytic) $L$-representation of $G$. Then $\Pi_{L'}=\Pi\otimes_{L} L'$ is a
 finite direct sum of irreducible admissible $L'$-Banach space (resp. locally analytic) representations of $G$.
 \end{lemma}

   This lemma yields the desired result, since if $\rm{End}_{L[G]}(\Pi)=L$, then $\rm{End}_{L'[G]}(\Pi_{L'})=L'$ and
   so $\Pi_{L'}$ has to be irreducible. We will prove the lemma in the locally analytic case, which is
   more difficult (see lemma 4.2 of \cite{Pa} for the Banach case). Fix $r\in (p^{-1},1)\cap p^{\mathbb{Q}}$ such that $\Pi_r'\ne 0$.

  \begin{lemma}\label{nonnul}
   For any nonzero $L'[G]$ closed subrepresentation $V$ of $\Pi_{L'}$ we have
   $V_r'\ne 0$.
  \end{lemma}

  \begin{proof}
  We follow some of the arguments
   in \cite{Pa}, lemma 4.2. Let $\Gamma=\rm{Gal}(L'/L)$ and let $r_{\gamma}$ be the continuous automorphism of
   $\Pi_{L'}$ defined by $r_{\gamma}(v\otimes c)=v\otimes \gamma(c)$. We claim that the natural map
   $\oplus_{\gamma\in \Gamma} r_{\gamma}(V)\to \Pi_{L'}$ is surjective. Its image
   $X$ is closed in $\Pi_{L'}$ (as the image of a map between admissible representations),
   stable by $G$ and so $X^{\Gamma}\subset \Pi$ is a closed $G$-stable subspace of $\Pi$.
   But $X^{\Gamma}\ne 0$, as for any $v\in X-\{0\}$ there exists $c\in L'$ such that
   $\sum_{\gamma\in \Gamma} r_{\gamma}(cv)\ne 0$ (by linear independence of characters).
   Thus $X^{\Gamma}=\Pi$ and so $X=\Pi_{L'}$. Now, by duality
   $\Pi_{L'}'$ is a submodule of $\oplus_{\gamma\in\Gamma} r_{\gamma}(V')$ and by flatness
   of $\mathcal{D}_r(H,L)$ over $\mathcal{D}(H,L)$, $\Pi_{r}'\otimes L'$ is a submodule of
   $\oplus_{\gamma\in\Gamma} r_{\gamma}(V'_r) \simeq \oplus_{\gamma\in \Gamma} (V'_r \otimes_{L', \gamma} L')$. In particular, $V_r'\ne 0$.
  \end{proof}

\begin{lemma}
  $\Pi_{L'}'$ contains a maximal strict coadmissible $\mathcal{D}(H,L')$-submodule stable by $G$.
\end{lemma}

\begin{proof}
 In this proof, coadmissible submodule will mean a coadmissible $\mathcal{D}(H,L')$-submodule of
 $\Pi'_{L'}$, stable by $G$. By dualizing the result of the previous lemma, it follows that
 for any coadmissible strict submodule $M$ of $\Pi_{L'}'$ we have $M_r\ne (\Pi_{L'})_r'$.
Suppose now that $(M_i)_{i\in I}$ is a chain of strict coadmissible submodules of $\Pi_{L'}'$ and let
$N$ be the closure in $\Pi_{L'}'$ of $\cup_{i} M_i$. Then $N$ is a coadmissible submodule of
$\Pi_{L'}'$ and we claim that $N\ne \Pi_{L'}'$. If we manage to prove this, we are done by Zorn's lemma.
As $(\Pi_{L'})'_r$ is finitely generated over $\mathcal{D}(H,L')_r$, the union of $(M_i)_r$ is a strict
submodule of $(\Pi_{L'})'_r$. On the other hand, $\cup_{i} (M_i)_r$ is a sub $\mathcal{D}_r(H,L)$-module of
$(\Pi_{L'})'_r$, thus it is closed (we recall that all submodules of a finitely generated module over a noetherian
Banach algebra are closed). On the other hand, by definition $\cup_{i} (M_i)_r$ is dense in
$N_r$. Consequently $N_r\subset \cup_{i} (M_i)_r$ and so $N_r\ne (\Pi_{L'})'_r$. Therefore
$N\ne \Pi_{L'}$, which finishes the proof of the lemma.

\end{proof}

  Dualizing the previous lemma, we deduce that $\Pi_{L'}$ contains an irreducible closed
  $L'[G]$-subrepresentation $V$. The argument in lemma \ref{nonnul} shows that
  $\Pi_{L'}$ is a quotient of $\oplus_{\gamma\in\Gamma} r_{\gamma}(V)$. Each of the
  representations $r_{\gamma}(V)$ is admissible and irreducible. We easily deduce that
  $\Pi_{L'}$ is a finite direct sum of irreducible admissible representations.

\end{proof}

\begin{corollary}
 1) Any absolutely irreducible admissible Banach space representation of a $\qp$-analytic group has a central character.

 2) Any absolutely irreducible admissible locally analytic representation of a $\qp$-analytic group has an infinitesimal
 character.
\end{corollary}

\subsection{Algebraic Fr\'echet division algebras and finiteness results}

 The aim of this section is to prove a non commutative version of a classical result, stating that
 a complete algebraic field extension of $\qp$ is finite over $\qp$. Combined with theorem \ref{alg},
 this will yield the finiteness property of $\rm{End}_{L[G]}(\Pi)$ when $\Pi$ is irreducible.
 Despite the rather innocent looking statement, we could not find a proof avoiding some
 pretty difficult results from
 non commutative algebra.

\begin{theorem}\label{Ban}
  Let $A$ be a Fr\'echet algebra over $L$. If $A$ is a division algebra and if any element of $A$ is algebraic
over $L$, then $A$ is finite dimensional over $L$.
\end{theorem}

\begin{proof}
Let $F_n$ be the set of those $a\in A$ that are killed by some polynomial $f\in L[X]$ of degree at most
$n$ and whose Gauss norm is between $1/n$ and $n$. The restriction on the Gauss norm of the polynomials killing the elements of $F_n$, the local compactness of
$L$ and the completeness of $A$ imply that $F_n$ is closed in $A$. By hypothesis, $A$ is the union of the $F_n$'s, so by Baire's lemma there exists
$n$ such that $F_n$ contains an open ball $B(x,r)$ of $A$. Fix such $n$. If $u,v\in F_n$, then
$(u^i)_{0\leq i\leq n}$ is linearly dependent over $L$ and thus so is the family $(u^iv)_{0\leq i\leq n}$. Therefore
$$\sum_{\sigma\in S_{n+1}} \varepsilon(\sigma) (u^{\sigma(0)}v \cdot u^{\sigma(1)}v\cdot...\cdot u^{\sigma(n)}v)=0.$$
That is, there exists a nonzero homogeneous polynomial in non commutative variables $p(u,v)$ such that
$p(u,v)=0$ for all $u,v\in F_n$. Let $a,b\in A$ be arbitrary elements. Since $B(x,r)\subset F_n$, for all $\lambda_1$, $\lambda_2\in L$ of sufficiently
large valuation we have $x+\lambda_1 a, x+\lambda_2 b\in F_n$ and so $p(x+\lambda_1 a, x+\lambda_2 b)=0$. But this remains
true for any $\lambda_1, \lambda_2$, as it is a polynomial identity. We deduce (using homogeneity of $p$) that $p(a,b)=0$ for any $a,b\in A$.
Therefore, $A$ is a polynomial-identity algebra. By a deep theorem of Kaplansky (\cite{Ka}), any polynomial-identity division
algebra is finite dimensional over its centre $Z$. Fix an algebraic closure $\overline{\qp}$ of $\qp$ and an embedding $\qp\to \overline{\qp}$. Since $Z$ is an algebraic field extension of $\qp$, this embedding extends to an embedding
$Z\to \overline{\qp}$ of $\qp$-vector spaces. By Krasner's lemma, $\overline{\qp}$ has countable dimension over
$\qp$, so $Z$ has at most countable dimension over $\qp$. But any Fr\'echet space of at most countable dimension
over $\qp$ is finite dimensional, as it easily follows from Baire's lemma. We deduce that $Z$ is finite dimensional over
$\qp$, hence that $A$ is finite dimensional over $\qp$. The result follows.
\end{proof}

  We are now able to prove the following finiteness result:

\begin{theorem}
  If $\Pi$ is an irreducible admissible $L$-Banach space (resp. locally $\qp$-analytic) representation of a $\qp$-analytic group $G$, then $\rm{End}_{L[G]}(\Pi)$
is a finite dimensional $L$-division algebra.
\end{theorem}

\begin{proof}

  The discussion following theorem \ref{dual1} shows that
  $\rm{End}_{L[G]}(\Pi)$ is a division algebra. In view of theorems
  \ref{alg} and \ref{Ban}, it is enough to prove that $\rm{End}_{L[G]}(\Pi)$ is naturally
  a Fr\'echet algebra. For Banach space representations, this is clear, since
  in this case $\rm{End}_{L[G]}(\Pi)$ is a closed subalgebra of $\rm{End}_{L}(\Pi)$,
  which is naturally a Banach algebra for the usual sup norm. Let us assume that
  $\Pi$ is an admissible locally analytic representation of $G$.  Fix a compact open subgroup $H$ of $G$, so the topological dual $\Pi'$ is a coadmissible $\mathcal{D}(H,L)$-module.
As $\Pi'\simeq \varprojlim_r \Pi'_r$ and $\mathcal{D}(H,L)\simeq \varprojlim_r \mathcal{D}_r(H,L)$, one has a natural isomorphism
of $L$-algebras $$\emorph_{\mathcal{D}(H,L)}(\Pi')\simeq \varprojlim_r \emorph_{\mathcal{D}_r(H,L)}(\Pi'_r).$$
Since $\Pi\to \Pi'$ induces an antiequivalence between admissible locally analytic representations of
 $H$ and coadmissible $\mathcal{D}(H,L)$-modules, we have a canonical isomorphism $\emorph_{\mathcal{D}(H,L)}(\Pi')\simeq
\emorph_{L[H]}(\Pi)^{\rm op}$. Note that $\emorph_{L}(\Pi'_r)$ is naturally a Banach algebra
for the sup norm (as $\Pi'_r$ is a Banach space) and $\emorph_{\mathcal{D}_r(H,L)}(\Pi'_r)$ is a closed subspace of it, so
again a Banach algebra. Combining these results with the observation that $\rm{End}_{L[G]}(\Pi)$
is a closed sub-algebra of $\rm{End}_{L[H]}(\Pi)$ yields the desired result. The induced structure of
Fr\'echet algebra on $\emorph_{L[G]}(\Pi)$ is independent of $H$, as if $H_1\subset H_2$ are two compact
open subgroups of $G$, then $\mathcal{D}(H_1,L)$ (respectively $\mathcal{D}_r(H_1,L)$) is a finite free module over
$\mathcal{D}(H_2,L)$ (respectively $\mathcal{D}_r(H_2,L)$).

\end{proof}

 We also note a rather useful consequence of the previous theorem. This has already been observed
 in \cite{Pa}, lemma 4.20. There are some extra-assumptions in that reference, precisely because of the
 lack of the previous result. However, the reader can check that the proof in loc.cit. only uses
 the fact that $\rm{End}_{L[G]}(\Pi)$ is finite dimensional over $L$ and that it also adapts to the case
 of locally analytic admissible representations.

\begin{corollary}
 Let $G$ be a $\qp$-analytic group and let $\Pi$ be an irreducible admissible
 Banach space (resp. locally $\qp$-analytic) representation of $G$. There exist a finite Galois extension
 $L'/L$ and finitely many absolutely irreducible admissible Banach space (resp. locally analytic) representations
 $\Pi_1,\Pi_2...,\Pi_s$ over $L'$ such that $$\Pi\otimes_{L} L'\simeq \Pi_1\oplus \Pi_2\oplus...\oplus \Pi_s.$$
\end{corollary}

\noindent
CMLS \'Ecole Polytechnique\\
UMR CNRS 7640\\
F--91128 Palaiseau cedex, France\\
{\it Adresse e-mail :} {\ttfamily gabriel.dospinescu@math.polytechnique.fr}
\\

\noindent
Laboratoire de Math\'ematiques de Versailles\\
UMR CNRS 8100\\
45, avenue des \'Etats Unis - B\^atiment Fermat\\
F--78035 Versailles Cedex, France\\
{\it Adresse e-mail:} {\ttfamily benjamin.schraen@math.uvsq.fr}

\end{document}